\documentclass[a4paper, 12pt]{article}

\usepackage{amsmath, amsthm, amssymb}
\usepackage{graphicx}
\usepackage{enumitem}
\usepackage{authblk}
\usepackage[textsize=small, textwidth=2cm, color=yellow]{todonotes}
\usepackage{thmtools}
\usepackage{thm-restate}

\usepackage[utf8]{inputenc}
\usepackage[russian, english]{babel}
\usepackage[T1]{fontenc}
\usepackage[a4paper, margin=2.5cm]{geometry}
\usepackage{needspace}

\usepackage{libertine} 
\usepackage{inconsolata} 

\usepackage{subcaption}
\usepackage[hyphens]{url}
\usepackage[linktoc = all, hidelinks, colorlinks, unicode=true, pagebackref=true]{hyperref}
\usepackage[capitalise, compress, nameinlink, noabbrev]{cleveref}
\hypersetup{linkcolor={black}, citecolor={blue!70!black}, urlcolor={blue!70!black}}

\declaretheorem[name=Theorem]{theorem}
\declaretheorem[name=Lemma, sibling=theorem]{lemma}

\declaretheorem[name=Conjecture, sibling=theorem]{conjecture}
\declaretheorem[name=Hypothesis]{hypothesis}
\declaretheorem[name=Problem, sibling=theorem]{problem}

\declaretheorem[name=Remark, style=remark]{remark}

\def\cqedsymbol{\ifmmode$\lrcorner$\else{\unskip\nobreak\hfil
\penalty50\hskip1em\null\nobreak\hfil$\lrcorner$
\parfillskip=0pt\finalhyphendemerits=0\endgraf}\fi} 

\let\le\leqslant
\let\ge\geqslant

\let\geq\geqslant

\title{Binary matroids and degree-boundedness for pivot-minors}

\author[1]{Rutger Campbell}
\author[2]{James Davies}
\author[3]{Robert Hickingbotham}

\affil[1]{University of Waterloo, Canada}
\affil[2]{Leipzig University, Germany}
\affil[3]{Université libre de Bruxelles, Belgium}
\date{}

\begin{document}

\maketitle

\begin{abstract}
    We prove that for every bipartite graph $H$ and positive integer $s$, the class of $K_{s,s}$-subgraph-free graphs excluding $H$ as a pivot-minor has bounded average degree.
    Our proof relies on the announced structure theorem of Geelen, Gerards, and Whittle for highly connected binary matroids.
    
    Along the way, we also prove that every $K_{s,t}$-free bipartite circle graph with $s\le t$ has a vertex of degree at most $\max\{2s-2, t-1\}$ and provide examples showing that this is tight.
\end{abstract}

\section{Introduction}

In this paper, we study degree-boundedness of graph classes defined by a forbidden pivot-minor.
A graph $H$ is a \emph{pivot-minor} of another graph $G$ if it is obtainable from $G$ by a sequence of pivots and vertex deletions; see \cref{sec:binary} for the formal definition. For a graph $H$, we say a graph is \emph{$H$-free} if it does not contain a subgraph isomorphic to $H$. Our main result is the following.

\begin{theorem}\label{main}
    For every bipartite graph $H$, there exists a function $f_H$ such that for every positive integer $s$, every $K_{s,s}$-free graph excluding $H$ as pivot-minor has average degree at most $f_H(s)$.
\end{theorem}

The condition that $H$ is bipartite is necessary. Since the class of bipartite graphs is closed under taking pivot-minors, every non-bipartite graph $H$ is excluded as a pivot-minor in a bipartite graph.
Thus, \cref{main} does not hold for any non-bipartite $H$ since there exists bipartite $C_4$-free graphs of arbitrarily large average degree (for example, for a prime power $q$, the point-line incidence graph of the projective plane $\text{PG}(2,q)$ is a bipartite $C_4$-free graph with all vertices of degree $q+1$). 

Following Du and McCarty \cite{du2024survey}, we say that a hereditary class of graphs is \emph{degree-bounded} if, for every positive integer $s$, the $K_{s,s}$-free graphs in the class have bounded average degree.
Classic results in this direction are R\"{o}dl's theorem (see~\cite{gyarfas1980induced,kierstead1996applications}) that graphs forbidding an induced tree are degree-bounded (see also~\cite{char2025edge,hunter2024k,scott2023polynomial}) and Kuhn and Osthus's \cite{kuhn2004induced} theorem that graphs forbidding the induced subdivisions of a graph are degree-bounded (see also~\cite{bourneuf2023polynomial,char2025edge,du2023induced,girao2023induced,hunter2024k}). In this context, \cref{main} says that pivot-minor-closed classes of graphs that do not contain all bipartite graphs are degree-bounded.
For more on degree-boundedness, see the recent survey of Du and McCarty \cite{du2024survey}.

Via fundamental graphs (as defined in \cite{bouchet1988graphic,oum2005rank}), proper pivot-minor-closed classes of graphs generalize proper minor-closed classes of binary matroids, which in turn generalize proper minor-closed classes of graphs via graphic matroids.
To date, only few structural results have been established for an arbitrary class of graphs defined by a forbidden pivot-minor. For example, it is known that such a class has the (strong) Erd\H{o}s-Hajnal property \cite{davies2023pivot}. Pivot-minor-closed classes of graphs extend and are closely related to vertex-minor-closed classes of graphs. See the recent survey of Kim and Oum \cite{kim2024vertex} for more on vertex-minors and pivot-minors.

Our proof of \cref{main} relies on the announced structure theorem of Geelen, Gerards, and Whittle \cite{geelen2015highly} for highly connected binary matroids forbidding a minor (see \cref{thm:mat_mc_str}), which is essentially equivalent to a structure theorem for bipartite graphs forbidding a fixed bipartite graph as a pivot-minor (see \cref{thm:bip_mc_str}).

A \emph{circle graph} is the intersection graph of a collection of chords in a circle.
Bipartite circle graphs are believed to play a fundamental role in the structure of pivot-minor-closed classes of bipartite graphs, analogous to the role that planar graphs play in the structure of minor-closed classes of graphs.
For instance, Oum's \cite{oum2009excluding} conjectured pivot-minor grid theorem states that for every bipartite circle graph $H$, every graph not containing $H$ as a pivot-minor has bounded rank-width. Bipartite circle graphs form a pivot-minor-closed class of graphs. Moreover, Geelen and Oum \cite{geelen2009circle} characterized bipartite circle graphs as the bipartite graphs forbidding the fundamental graphs of the matroids $M(K_{3,3})$, $M(K_5)$, and $F_7$. This is equivalent to Tutte's~\cite{tutte1959matroids} theorem that a binary matroid is the graphic matroid of a planar graph if and only if it forbids $M(K_{3,3})$, $M(K_5)$, $F_7$, and their duals as a minor.

Along the way to proving \cref{main}, we show the following using de Fraysseix's theorem~\cite{de1981local} (see \cref{Fray}).

\begin{theorem}\label{circle}
    Every $K_{s,t}$-free bipartite circle graph with $s\le t$ has a vertex of degree at most $\max\{2s-2,t-1\}$.
\end{theorem}

\cref{circle} is best possible in the sense that for every pair of positive integers $s,t$, there is a $K_{s,t}$-free bipartite circle graph with minimum degree equal to $\max\{2s-2,t-1\}$ (see \cref{fig:Fundamentalgraphs}).
For circle graphs in general, a result of Fox and Pach \cite{fox2010separator} implies a linear upper bound on their minimum degree.

Our proof of \cref{main} proceeds in three steps. In \cref{sec:fun}, we prove degree-boundedness for fundamental graphs of graphic matroids and also for their bipartite complements. Along the way in this section, we also deduce \cref{circle}. In \cref{sec:rank}, we show that every $K_{s,s}$-free graph with large average degree contains an induced $C_4$-free bipartite subgraph with large average degree and large rank-connectivity. Finally, in \cref{sec:binary}, we combine these results with the binary matroid structure theorem of Geelen, Gerards, and Whittle \cite{geelen2015highly} to obtain \cref{main}. We conclude with some further open problems in \cref{sec:open}.

\section{Fundamental graphs}\label{sec:fun}

For a connected graph $G$ with a spanning tree $T$, the \emph{fundamental matrix of $G$ with respect to $T$}
is the binary matrix $D$ with rows indexed by $E(T)$ and columns indexed by $E(G)-E(T)$, where, for $i\in E(T)$ and $j\in E(G)-E(T)$, the entry $D_{i,j}$ is equal to $1$ if the cycle in $T\cup\{j\}$ contains $i$, and is equal to $0$ if it does not.
The bipartite graph with biadjacency matrix $D$ between $E(T)$ and $E(G)-E(T)$ is called the \emph{fundamental graph of $G$ with respect to $T$}. Note that $G$ is allowed to have multi-edges. If a bipartite graph can be constructed in this way, then we call it a \emph{graphic fundamental graph} (as it is the fundamental graph of a graphic matroid).
Note that graphic fundamental graphs form a hereditary class of graphs as deletion of vertices in $E(T)$ corresponds to their contraction in $G$, while deletion of vertices in $E(G)-E(T)$ corresponds to their deletion in $G$.
As the bipartite complements of members in a hereditary class of bipartite graphs themselves form a hereditary class, the bipartite complements of graphic fundamental graphs also form a hereditary graph class.

We begin by showing that graphic
fundamental graphs are degree-bounded.

\begin{lemma}\label{lem:fun}
    For positive integers $s\le t$, each $K_{s,t}$-free graphic fundamental graph contains a vertex of degree at most $\max\{2s-2,t-1\}$.
\end{lemma}

\begin{proof}
    Suppose otherwise, then there exists a fundamental graph $H$ of a connected graph $G$ with respect to a spanning tree $T$ such that $H$ has minimum degree at least $\max\{2s-1,t\}$ and is $K_{s,t}$-free.

    Choose a vertex $r$ of $T$ and a leaf $\ell$ of $T$ whose distance from $r$ is maximised.
    Let $P$ be the path in $T$ between $\ell$ and $r$, and let $p$ be the vertex adjacent to $\ell$ in $T$.
    Let $F$ be the set of edges in $E(G)\backslash E(T)$ incident to $\ell$.
    Since every edge of $G$ has degree at least $t$ in $H$, it follows that $|F|\ge t$.
    Each $f\in F$ is adjacent in $H$ to $E(P_f)$ for some path $P_f$ in $T$ starting at $\ell$. By the maximality of $\ell$, it follows that $P_f\cap P$ is a subpath of $P$ of length at least $|E(P_f)|/2 $ starting at $\ell$ for each $f\in F$.
    As $|E(P_f)|\ge 2s-1$ for each $f\in F$, it follows that $\bigcap_{f\in F} P_f$ is a subpath $P^*$ of $P$ that starts at $\ell$ and is of length at least $s$.
    But then $E(P^*)\cup F$ induces a complete bipartite graph in $H$ that contains $K_{s,t}$ as a subgraph, a contradiction.
\end{proof}

\Cref{circle} now immediately follows from \cref{lem:fun} and the following theorem of de Fraysseix \cite{de1981local}.

\begin{theorem}[de Fraysseix \cite{de1981local}]\label{Fray}
    A bipartite graph is a circle graph if and only if it is a fundamental graph of a planar graph.
\end{theorem}

The minimum degree bounds in \cref{lem:fun} and in \cref{circle} are tight. When $t-1>2s-2$, the graph $K_{t-1,t-1}$ has minimum degree $t-1$ and is a $K_{s,t}$-free fundamental graph of a planar graph (see \cref{fig:Fundamentalgraphs} (a)).
When $2s-2\geq t$, the $(s-1)$-blow up of the $6$-cycle (the graph obtained from the $6$-cycle by replacing each vertex by an independent set of $s-1$ vertices where two vertices are adjacent if their original ones are) has minimum degree $2s-2$ and is a $K_{s,t}$-free fundamental graph of a planar graph (see \cref{fig:Fundamentalgraphs} (b)).

\begin{figure}
    \centering
    \begin{subfigure}[t]{0.4\textwidth}
        \centering
        \includegraphics{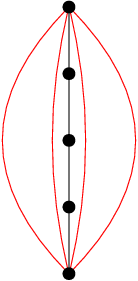}
        \caption{$K_{4,4}$ as a fundamental graph of a planar graph.}
    \end{subfigure}
    \hspace{0.05\textwidth}
    \begin{subfigure}[t]{0.35\textwidth}
        \centering
        \includegraphics{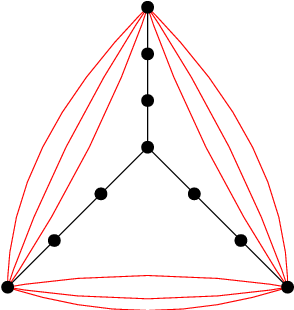}
        \caption{The $3$-blow up of $C_6$ as a fundamental graph of a planar graph.}        
    \end{subfigure}
    \caption{Representations of bipartite circle graphs as fundamental graphs of planar graphs: black edges correspond to the spanning tree $T$, red edges are the non-tree edges $E(G)-E(T)$.}\label{fig:Fundamentalgraphs}
\end{figure}

For a bipartite graph $G=(A,B)$, the \emph{bipartite complement} of $G$ is the bipartite graph with sides $A$ and $B$ that has an edge $e\in A\times B$ precisely when $G$ does not.
We also need to examine the bipartite complements of the fundamental graphs of graphs.
Before this, we require a simple lemma on trees.

\begin{lemma}\label{lem:tree}
    Let $s$ be a positive integer and $T$ be a tree with at least $5s$ edges.
    Then either there exists an edge $e$ such that the two trees of $T-e$ have at least $s$ edges, or there exists a vertex $v$ and three edge-disjoint subtrees $T_1,T_2,T_3$, each with at least $s$ edges and with $V(T_1)\cap V(T_2)= V(T_1)\cap V(T_3) = V(T_2)\cap V(T_3) = \{v\}$.
\end{lemma}

\begin{proof}
    Choose an arbitrary root $r$ of $T$. For two vertices $u,v$ of $T$, we say that $u$ is a \emph{descendant} of $v$ when $v$ is contained on the unique path between $r$ and $u$. Note that $u$ is a descendant of itself.
    For a vertex $v$ of $T$, we let $T(v)$ be the subtree of $T$ induced by the descendants of $v$.

    Choose a vertex $v$ of $T$ with $|E(T(v))|\ge s$, and subject to this, with distance from $r$ maximised.
    Such a vertex exists since $r$ is a candidate.
    
    Suppose first that $|E(T(v))| \ge 3s$.
    Let $c_1, \ldots ,c_k$ be the children of $v$ (the descendents of $v$ that are adjacent to $v$).
    By the choice of $v$, we have that $|E(T(c_1))|, \ldots , |E(T(c_k))| < s$.
    In particular, $|E(T(c_1)) \cup \{c_1v\}|, \ldots , |E(T(c_k)) \cup \{c_kv\}| \le s$.
    Therefore, we can choose two edge-disjoint subtrees $T_1,T_2$ of $T(v)$ containing the vertex $v$, such that $s\le |E(T_1)|, |E(T_2)| \le 2s$ and $E(T_1), E(T_2)$ are each the union of some collection from $E(T(c_1)) \cup \{c_1v\}, \ldots , E(T(c_k)) \cup \{c_kv\}$.
    Let $T_3= T\backslash (T_1 \cup T_2)$, then $T_3$ is also a subtree of $T$ that is edge-disjoint from $T_1,T_2$ and contains the vertex $v$.
    Furthermore, $|E(T_3)|= |E(T)|-|E(T_1)|- |E(T_2)| \ge 5s - 2s -2s =s$, as desired.

    We may therefore now assume that $|E(T(v))| < 3s$.
    In particular, this implies that $v\not= r$.
    Let $p$ be the parent of $v$ (the vertex adjacent to $v$ that has $v$ as a descendent).
    Since $|E(T)|\ge 5s$, it now follows that the two trees of $T -pv$ have at least $s$ edges, as desired.
\end{proof}

\begin{lemma}\label{lem:cofun}
    Every $K_{s,s}$-free bipartite complement of a graphic fundamental graph contains a vertex of degree at most $5s-1$.
\end{lemma}

\begin{proof}
    Suppose otherwise, then there exists a graph $H$ that is the bipartite complement of a fundamental graph of a connected graph $G$ with respect to a spanning tree $T$ such that $H$ has minimum degree at least $5s$ and is $K_{s,s}$-free.

    Since $H$ has minimum degree at least $5s$, it follows that $T$ has at least $5s$ edges. Then by \cref{lem:tree}, either there exists an edge $e$ of $T$ such that the two trees of $T-e$ have at least $s$ edges, or there exists a vertex $v$ of $T$ and three edge-disjoint subtrees $T_1,T_2,T_3$, each with at least $s$ edges and with $V(T_1)\cap V(T_2)= V(T_1)\cap V(T_3) = V(T_2)\cap V(T_3) = \{v\}$.

    Suppose first that there exists an edge $e$ of $T$ such that the two trees $T_1,T_2$ of $T-e$ have at least $s$ edges.
    By the minimum degree of $H$, there is some $F\subseteq E(G)\backslash E(T)$ with $|F|\ge 2s$ such that every $f\in F$ is adjacent to $e$ in $H$. Observe that every such $f\in F$ has endpoints within either $V(T_1)$ or $V(T_2)$. In particular, in $H$, each such $f$ is either adjacent to every vertex of $E(T_1)$, or every vertex of $E(T_2)$.
    Without loss of generality, there are at least $s$ vertices of $F$ adjacent to every vertex of $E(T_1)$, contradicting the fact that $H$ is $K_{s,s}$-free.

    We may therefore assume that there exists a vertex $v$ of $T$ and three edge-disjoint subtrees $T_1,T_2,T_3$, each with at least $s$ edges and with $V(T_1)\cap V(T_2)= V(T_1)\cap V(T_3) = V(T_2)\cap V(T_3) = \{v\}$.
    Observe that for each $e\in E(G)\backslash E(T)$, the path $P_e$ contained in $T$ between the end points of $e$ is contained in one of $E(T_1)\cup E(T_2)$, $E(T_1)\cup E(T_3)$, $E(T_2)\cup E(T_3)$.
    In particular, in $H$, each $e\in E(G)\backslash E(T)$ is adjacent to all of the vertices of at least one of $E(T_1), E(T_2), E(T_3)$.
    Since $|E(G)\backslash E(T)| \ge 3s$, by the pigeonhole principle, we find $K_{s,s}$ as a subgraph of $H$, a contradiction.
\end{proof}

The following theorem gives degree-boundedness within a class that appears in the structure of bipartite graphs forbidding a bipartite graph as a pivot-minor.

\begin{theorem}\label{lem:struct}
    Let $n,s$ be positive integers.
    Let $G=(A,B)$ be a $K_{s,s}$-free bipartite graph with partitions $A_1,\dots, A_n$ and $B_1,\dots, B_n$ of $A$ and $B$ such that,
    for all $i,j\in\{1,\dots, n\}$, the graph induced on $A_i\cup B_j$ is either a graphic fundamental graph or the bipartite complement of a graphic fundamental graph.
    Then, $G$ has average degree at most $10n^2s$.
\end{theorem}

\begin{proof}
    For each pair $i,j\in\{1,\dots, n\}$, let $G_{i,j}$ be the subgraph induced on $A_i\cup B_j$.
    Then, $G_{i,j}$ is either a graphic fundamental graph or the bipartite complement of a graphic fundamental graph. By the handshaking lemma, it follows from \cref{lem:fun}, \cref{lem:cofun}, and the fact that graphic fundamental graphs and bipartite complements of graphic fundamental graphs are hereditary that $|E(G_{i,j})|\le 10s|V(G_{i,j})| \le 10s|V(G)|$.
    Since $G$ is the union of each $G_{i,j}$, it follows that $$|E(G)| = \sum_{i=1}^n \sum_{j=1}^n |E(G_{i,j})| \le \sum_{i=1}^n \sum_{j=1}^n 10s |V(G)| = 10n^2s |V(G)|,$$ as desired.
\end{proof}

\section{Induced subgraphs with large average degree}\label{sec:rank}

In this section, we show that every $K_{s,s}$-free graph with large average degree contains an induced bipartite $C_4$-free subgraph with large average degree and large rank-connectivity (see \cref{lem:rankconnectmain}).
Having large rank-connectivity is important for us to be able to apply the structure theorem of Geelen, Gerards, and Whittle \cite{geelen2015highly} for highly connected binary matroids. Let $G$ be a graph. For a partition $(A,B)$ of $V(G)$, we say the \emph{cut-rank} of $(A,B)$ is the rank of the binary adjacency matrix between $A$ and $B$, that is the $A\times B$-matrix $D$ over $\mathbb{Z}_2$ where for $a\in A$ and $b\in B$ we have $D_{ab}=1$ precisely when $a$ is adjacent to $b$. We say $(A,B)$ is a \emph{separation of rank $\ell$} when it has rank $\ell$ and $|A|,|B| > \ell$. We say a graph $G$ is \emph{$k$-rank-connected}
if there is no separation of rank $\ell$ for any $\ell\in\{0,1,\dots,k-1\}$. We will see the connection between rank-connectivity and the connectivity of binary matroids in Section~\ref{sec:binary}.

The following fundamental result of McCarty \cite{mccarty2021dense} is very useful in the study of degree-boundedness since it reduces the problem of showing that a hereditary graph class is degree-bounded to just showing that the bipartite $C_4$-free graphs in the class have bounded average degree.

\begin{theorem}[McCarty \cite{mccarty2021dense}]\label{mccarty}
    There is a function $f$ such that for all positive integers $d,s$, every $K_{s,s}$-free graph with average degree at least $f(d,s)$ has a bipartite $C_4$-free induced subgraph with average degree at least $d$.
\end{theorem}

Next we handle the reduction to large rank-connectivity.
For vertex connectivity, this was proved by Mader \cite{mader1972existenz}.
A graph $G$ is \emph{$k$-connected} if it has more than $k$ vertices and $G-X$ is connected for every $X\subseteq V(G)$ with $|X|<k$.

\begin{theorem}[Mader \cite{mader1972existenz}]\label{mader}
    For every positive integer $k$, every graph with average degree at least $4k$ has a $(k + 1)$-connected induced subgraph with more than $2k$ vertices.
\end{theorem}

Note that a $(k + 1)$-connected subgraph has minimum (and thus average) degree at least $k$.
In $C_4$-free graphs, graphs with huge connectivity have large rank-connectivity.
We have optimized the following lemma for simplicity rather than to obtain the best bounds.

\begin{lemma}\label{lem:rankconnect1}
    Every $2^k$-connected $C_4$-free graph is $k$-rank-connected.
\end{lemma}
\begin{proof}
    Trivially this is true for $k=1$, so we may assume that $k>1$.
    Let $G$ be a $2^k$-connected $C_4$-free graph and suppose, for the sake of contradiction, that $G$ contains a separation $(A,B)$ of rank $k'<k$.
    Since $G$ is $2^k$-connected, we have that $|V(G)| > 2^k$, in particular, we may assume without loss of generality that $|A|>  2^{k'}$.
    Note also that $|B|>k'$ by definition of a separation of rank $k'$.
    
    Since the rank of the binary adjacency matrix between $A$ and $B$ is at most $k'$, it follows that there exists a partition of $A$ into subsets $A_1, \ldots , A_{2^{k'}}$ such that for every $b\in B$, there exists some $I\subseteq \{1, \ldots , 2^{k'}\}$ so that $N(b)\cap A = \bigcup_{i\in I} A_i$.
    Let $S$ be the integers $1\le i \le 2^{k'}$ such that $|A_i| = 1$, and
    let $T$ be the integers $1\le i \le 2^{k'}$ such that $|A_i| > 1$.
    Let $A_S = \bigcup_{i\in S} A_i$. Then $|A_S|=|S|\le 2^{k'} <|A|$.
    Let $A_T = \bigcup_{i\in T} A_i$, and 
    let $B_T$ be the vertices $b\in B$ such that $A_i\subseteq N(b)$ for some $i\in T$.
    Since $G$ is $C_4$-free, we have that for each $i\in T$, there is a single vertex $b\in B$ with neighbours in $A_i$.
    This in particular implies that the rank of the binary adjacency matrix between $A_T$ and $B$ is equal to $|B_T|$, and therefore that $|B_T|\le k' < |B|$.
    Let $X=A_S\cup B_T$. Then $|X|\le 2^{k'} + k' <2^k$.
    As $|A_S|<|A|$ and $|B_T|<|B|$, it follows that both $A\backslash X$ and $B \backslash X$ are non-empty.
    Observe further that $X$ is a vertex cut separating $A\backslash X$ from $B \backslash X$ in $G$ since for every $b\in B \backslash X = B\backslash B_T$, we have that $N(b)\cap A \subseteq A_S \subseteq X$.
    This contradicts the fact that $G$ is $2^k$-connected, and so the lemma follows.
\end{proof}

Combining \cref{mader} with \cref{lem:rankconnect1}, we obtain the following. 

\begin{lemma}\label{lem:rankconnect2}
    For every positive integer $k$, each $C_4$-free graph with average degree at least $2^{k+2}$ contains a $C_4$-free $k$-rank-connected induced subgraph with average degree at least $2^k$.
\end{lemma}

From \cref{mccarty} and \cref{lem:rankconnect2}, we obtain the following, which is the main result of this section.

\begin{lemma}\label{lem:rankconnectmain}
    There is a function $f$ such that for all positive integers $d,k,s$, each $K_{s,s}$-free graph with average degree at least $f(d,k,s)$ has a bipartite $k$-rank-connected $C_4$-free induced subgraph with average degree at least $d$.
\end{lemma}

\section{Pivot-minors and binary matroids}\label{sec:binary}

The proof of Theorem~\ref{main} relies on the following structural hypothesis which follows from the announced structure theorem of Geelen, Gerards, and Whittle~\cite{geelen2015highly} for highly connected binary matroids (see \cref{thm:mat_mc_str}).

\begin{hypothesis}\label{thm:bip_mc_str}
    For each bipartite graph $H$, there exist positive integers $k,t$ so that, for every $k$-rank-connected bipartite graph $G=(A,B)$ that does not contain $H$ as a pivot-minor, there are partitions $A_1,\dots, A_t$ and $B_1,\dots, B_t$ of $A$ and $B$ respectively such that, for all $i,j \in\{1,\dots, t\}$, the graph induced on $A_i\cup B_j$ is either a graphic fundamental graph or the bipartite complement of a graphic fundamental graph.
\end{hypothesis}
\cref{main} follows immediately from \cref{lem:rankconnectmain,thm:bip_mc_str,lem:struct}.

The rest of this section is devoted to explaining how the structure theorem of Geelen, Gerards, and Whittle~\cite{geelen2015highly} (see \cref{thm:mat_mc_str}) implies \cref{thm:bip_mc_str}.
We remark that this connection is well known within the matroid community.
For a more comprehensive introduction to matroids, see~\cite{WIMatroid, Oxley}.

Let $A$ be a matrix over the binary field, with columns indexed by a set $E$.
Let $\mathcal{C}\subseteq 2^E$ correspond to collections of columns of $A$ that form minimal linear dependencies;
that is to say $X\subseteq E$ is in $\mathcal{C}$ when the columns of $A$ indexed by $X$ form a linear dependency, but no proper subset of these columns are linearly dependent.
The pair $M=(E, \mathcal{C})$ is a \emph{binary matroid}.
Take some $B\subseteq E$ whose corresponding columns are maximally independent (a \emph{basis} of $M$).
By row operations and permuting columns, we may assume that $A$ is a $B\times E$ matrix of the form
$[I|D]$ where $D$ is a $B\times (E-B)$ binary matrix and $I$ is the $B\times B$ identity matrix.
The bipartite graph with biadjacency matrix $D$ between $B$ and $(E-B)$ is called the \emph{fundamental graph of $M$ with respect to $B$}.
Note in particular that every bipartite graph is a fundamental graph of some binary matroid.

Let $G$ be a graph with a spanning tree $T$ and let $D$ be the fundamental matrix of $G$ with respect to $T$.
The \emph{graphic matroid of $G$} is the binary matroid with representation $[I|D]$,
while the \emph{cographic matroid of $G$} is the binary matroid with representation $[I|D^\top]$. Note that these two matroids have the same bipartite fundamental graph but where the sides of the bipartition are swapped (according to $E(T)$ being a basis of the graphic matroid of $G$ and $E(G)\backslash E(T)$ being a basis of the cographic matroid of $G$).
A \emph{(co)graphic matroid} is a matroid that can be constructed from a graph in this way. As a consequence, the fundamental graphs of graphic and cographic matroids
are isomorphic to graphic fundamental graphs.

We remark that more generally, the fundamental graphs of a binary matroid $M$ are the same as the fundamental graphs of its dual matroid $M^*$. Cographic matroids being the duals of graphic matroids are a special case of this. More generally, the fundamental graph of a binary matroid $M$ with respect to a basis $B$ is the same as the fundamental graph of the dual $M^*$ with respect to the basis $E(M)-B$ (just with the sides of the bipartition swapped).

Consider a binary matroid $M=(E,\mathcal{C})$ with a basis $B$ and representation given by the binary $B\times E$-matrix $[I|D]$.
For $x\in B$ and $y\in E-B$ where the $x$-entry of column $y$, namely $D_{x,y}$, is $1$, consider using $x$-row to eliminate all other non-zero entries of column $y$.
This transformation looks like the following:
\[
\left[\begin{array}{cc|cc}
1&0 & 1&\beta\\
0&I & \alpha& D'
\end{array}\right]
\longrightarrow
\left[\begin{array}{cc|cc}
1&0 		& 1&\beta\\
-\alpha&I 	& 0& D'-\alpha\beta
\end{array}\right]
.\]
Note these row operations do not change the linear dependencies of the columns, and so we have that $B-x+y$ is a basis and, after swapping columns $x$ and $y$, $M$ has representation
\[\left[\begin{array}{cc|cc}
1&0 		& 1&\beta\\
0&I 	& -\alpha& D'-\alpha\beta
\end{array}\right]
.\]
This matrix transformation corresponds to a change of the basis with which we consider our representation.
Consider the bipartite graphs $G_1$ and $G_2$ on vertex set $E$ with $G_1$ given by the $B\times (E-B)$ biadjacency matrix $D_1=\left[\begin{array}{cc}
1&\beta\\
\alpha& D'
\end{array}\right]$ and, taking $B'=B-x+y$, $G_2$ given by the $B'\times (E-B')$ biadjacency matrix $D_2=\left[\begin{array}{cc}
1&\beta\\
-\alpha& D'-\alpha\beta
\end{array}\right]$.
Note that for $i\in B-x$ and $j\in E-B-y$, we have $(\alpha\beta)_{i,j}=1$ precisely when $i$ is adjacent to $y$ and $j$ is adjacent to $x$.
So $G_2$ is obtained from $G_1$ by complementing the edges between $N(x)-y$ and $N(y)-x$ and then swapping the labels $x$ and $y$.
In general, for a (not necessarily bipartite) graph $G$, we \emph{pivot} the edge $xy$ by: (i) taking $V_1=N(x)-N(y)-y$, $V_2=N(y)-N(x)-x$, $V_3=N(x)\cap N(y)$; (ii) complementing the edges between $V_1$ and $V_2$, between $V_2$ and $V_3$, between $V_3$ and $V_1$; (iii) and then swapping the labels $x$ and $y$ (see Figure~\ref{fig:pivoting}). Note that, in a bipartite graph, pivoting gives the bipartite graph previously described.
If one graph can be obtained from another graph by a sequence of pivots, then we say that they are \emph{pivot-equivalent}.
By the above, we have that each graph that is pivot-equivalent to a fundamental graph of a matroid is also a fundamental graph of the same matroid, but with respect to a different basis. 

\begin{figure}[t]
        \centering
        \includegraphics{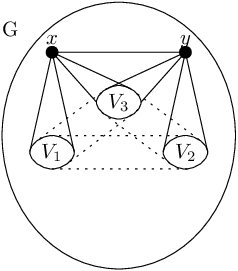}
        \caption{Pivoting the edge $xy$ in $G$ corresponds to edge-complementation within the three dashed regions and swapping the labels of $x$ and $y$.}\label{fig:pivoting}
\end{figure}

In a matroid with representation given by the binary $B\times E$-matrix $[I|D]$, we can \emph{contract} an element of $B$ by deleting the corresponding row and column, and \emph{delete} an element of $E-B$ by deleting the corresponding column.
To contract or delete other elements, we can change the basis of our representation, which as discussed, has a corresponding pivot-equivalent fundamental graph.
The result of a sequence of contractions and deletions is a \emph{minor} of the original matroid.
Thus a pivot-minor of a fundamental graph of a binary matroid $M$ corresponds to a fundamental graph of a minor of $M$.

For binary $r\times s$-matrices $D_1,D_2$, we say that $D_1$ is a \emph{$p$-perturbation} of $D_2$ when the matrix $D_1-D_2$ has rank at most $p$. For a matrix $D$, let $D[S,T]$ denotes the submatrix of $D$ restricted to rows indexed by $S$ and columns indexed by $T$. 
\begin{lemma}\label{lem:pert_to_part}
	For all positive integers $p,r,s$, if $C$ is a binary $r\times s$-matrix with rank $p$,
	then for $t= 2^p$, there is a partition $A_1,\dots,A_t$ of $[r]$ and a partition $B_1,\dots,B_t$ of $[s]$ so that for each $i,j\in\{1,\dots,t\}$ the  submatrix $C[A_i,B_j]$ is either all $0$'s or all $1$'s.
\end{lemma}

\begin{proof}
	As $C$ has rank $p$, the row- and column-spaces of $C$ are binary vector space of dimension $p$.
    So, there are at most $2^p=t$ distinct row vectors, and similarly at most $2^p=t$ distinct column vectors.
    Therefore, there is a partition $A_1,\dots,A_t$ of $[r]$ where for each $1\le i \le t$, the row vectors indexed by elements of $A_i$ are all equal.
    Similarly, there is a partition $B_1,\dots,B_t$ of $[s]$ where for each $1\le i \le t$, the column vectors indexed by elements of $B_i$ are all equal.
	It follows that for each $i,j\in\{1,\dots,t\}$ the  submatrix $C[A_i,B_j]$ has all rows and all columns the same and is therefore a constant matrix, as desired.
\end{proof}

For bipartite graphs $G_1$ and $G_2$ on the same vertex set and with the same bipartition $(A,B)$, we say that $G_1$ is a \emph{$p$-perturbation} of $G_2$ when the binary biadjacency matrix $D_1$ of $G_1$ is a $p$-perturbation of the binary biadjacency matrix $D_2$ of $G_2$.
As a corollary of Lemma~\ref{lem:pert_to_part}, we have the following:
\begin{lemma}\label{lem:pert_to_part_comp}
	For every non-negative integer $p$, if $G_1$ and $G_2$ are graphs with bipartition $(A,B)$, and $G_1$ is a $p$-perturbation of $G_2$, then there is a partition $A_1,\dots,A_t$ of $A$ and a partition $B_1,\dots,B_t$ of $B$ where $t=2^p$ such that, for all $i,j\in\{1,\dots,t\}$, the subgraph $G_1[A_i\cup B_j]$ is either equal to $G_2[A_i\cup B_j]$ or its bipartite complement.
\end{lemma}
For binary matroids $M_1$ and $M_2$ with the same rank and on the same ground set, 
we say $M_1$ is a \emph{$p$-perturbation} of $M_2$ when, up to column permutations, there is a representation $[I|D_1]$ of $M_1$ and a representation $[I|D_2]$ of $M_2$ such that $D_1$ is a \emph{$p$-perturbation} of $D_2$.
We remark that in~\cite{geelen2015highly}, $p$-perturbations are defined a bit differently allowing more general representations than those of the form $[I|D]$, however by considering row operations, it is easy to see that these definitions are equivalent.

Let $M$ be a binary matroid with ground set $E$ and binary representation $[I|D]$ with respect to some basis $B$.
For $X\subseteq E$ and taking $X_B=B\cap X$, $Y_B=B-X$, $X_C=X-B$, and $Y_C=E-B-X$,
define $\lambda(X)=\text{rk}(D[X_B,Y_C])+\text{rk}(D[Y_B,X_C])$.
For the bipartite graph $G$ with the $B\times (E-B)$ biadjacency matrix $D$, this is equivalent to the cut-rank of $(X, E(G)-X)$, as defined in Section~\ref{sec:rank}. 
A binary matroid $(E,\mathcal{C})$ is \emph{$k$-connected}, when for each $\ell\in\{1,\dots,k-1\}$, there is no $X$ with $|X|,|E-X|\geq \ell$ and $\lambda(X)<\ell$.
As this definition coincides with the definition of rank-connectivity (except that it is 1 less), we have the following which has also been proven for instance in \cite{bouchet1989connectivity,oum2005rank}.

\begin{lemma}\label{rmk:con}
    Let $M$ be a binary matroid and $G$ be a fundamental graph of $M$. For each positive integer $k$,  $M$ is $k$-connected if and only if $G$ is $(k+1)$-rank-connected.
\end{lemma}

The following result has been announced by Geelen, Gerards, and Whittle \cite{geelen2015highly}.

\begin{hypothesis}[Geelen, Gerards, and Whittle \cite{geelen2015highly}]\label{thm:mat_mc_str}
	Let $\mathcal{M}$ be a proper minor-closed class of binary matroids.
	Then there exist positive integers $k,p$ such that each $k$-connected member of $\mathcal{M}$ is a $p$-perturbation of a graphic matroid or a cographic matroid.
\end{hypothesis}

We remark that the version given in \cite{geelen2015highly} uses vertical connectivity rather than Tutte connectivity as in the present paper. This is stronger than the above version since every vertically $k$-connected matroid is Tutte $k$-connected.
For further discussion of vertical and Tutte connectivity in matroid structure theorems, see \cite{grace2018perturbations}.

The class of bipartite graphs excluding a bipartite graph $H$ as a pivot-minor corresponds to the fundamental graphs of the binary matroids that exclude the binary matroids with $H$ as their fundamental graph (which is just some binary matroid and its dual). As the fundamental graphs of graphic matroids and cographic matroids are graphic fundamental graphs, \cref{thm:bip_mc_str}, now is a direct implication of \cref{thm:mat_mc_str}, \cref{rmk:con}, and \cref{lem:pert_to_part_comp}.
Thus, \cref{main} holds assuming \cref{thm:mat_mc_str}.

\section{Open problems}\label{sec:open}

A \emph{subdivision} of a graph $G$ is a graph obtainable from $G$ by replacing edges by possibly
longer paths.
Kuhn and Osthus \cite{kuhn2004induced} proved that, for every graph $H$, the class of graphs that do not contain an induced subgraph isomorphic to a subdivision of $H$ is degree-bounded.
We conjecture a common strengthening of both Kuhn and Osthus's \cite{kuhn2004induced} theorem and \cref{main}.
An \emph{odd subdivision} of a graph $G$ is a graph obtainable from $G$ by replacing each edge
by a path of odd length.

\begin{conjecture}\label{odd}
    For every bipartite graph $H$ and positive integer $s$, the $K_{s,s}$-free graphs with no induced odd subdivision of $H$ have bounded average degree.
\end{conjecture}

The bipartiteness condition on $H$ is necessary as otherwise no bipartite graph would be forbidden.
Observe that \cref{odd} implies Kuhn and Osthus's \cite{kuhn2004induced} theorem since every graph has a subdivision that is bipartite.
It also implies \cref{main} since every odd subdivision of a graph $H$ contains $H$ as a pivot-minor.
We remark that \cref{odd} also relates to a conjecture of Scott and Seymour \cite{scott2020survey} that $K_{s,s}$-free graphs with no induced odd subdivision of a graph $H$ (with $H$ not necessarily bipartite) have bounded chromatic number.

Our proof of \cref{main} relies on Geelen, Gerards, and Whittle's \cite{geelen2015highly} structure theorem for highly connected binary matroids, whose proof is incredibly long and complicated.
It is desirable to have a simpler proof of \cref{main} that does not rely on this structure theorem.
Indeed, a more elementary proof of \cref{main} may lead to progress on \cref{odd}.

Following McCarty's \cite{mccarty2021dense} fundamental theorem (see \cref{mccarty}) on degree-boundedness, further work was done optimizing the resulting bounds in degree-bounded classes \cite{bourneuf2023polynomial,du2023induced,girao2023induced,hunter2024k}.
The current best such bounds are due to Hunter, Milojevi{\'c}, Sudakov, and Tomon \cite{hunter2024k}.
In particular, the bounds in \cref{main} are polynomial \cite{girao2023induced,hunter2024k}, but no uniform bound on the degree of the polynomial are known. 

Du and McCarty \cite{du2024survey} recently conjectured that every proper vertex-minor-closed graph class is linearly degree-bounded.
We further conjecture the natural pivot-minor extension of this conjecture.

\begin{conjecture}\label{linear}
    For each bipartite graph $H$, there exists a constant $C_H>0$ such that every $K_{s,s}$-free graph excluding $H$ as pivot-minor has average degree at most $C_Hs$.
\end{conjecture}

We remark that Hunter, Milojevi{\'c}, Sudakov, and Tomon \cite{hunter2025c_4} recently showed that degree-bounded classes with the density-Erdős-Hajnal property (see \cite{hunter2025c_4}) are linearly degree-bounded.
So, this could provide one approach to \Cref{linear}, and it is encouraging that these classes do have the related strong Erdős-Hajnal property \cite{davies2023pivot}.

\cref{circle} gave a tight upper bound for the minimum degree of a $K_{s,t}$-free bipartite circle graph.
However, the resulting average degree bound is not known to be tight.

\begin{problem}
    What is the maximum average degree of $K_{s,t}$-free (bipartite) circle graphs?
\end{problem}

Our proof of \cref{circle} relied on de Fraysseix's \cite{de1981local} characterization of bipartite circle graphs as fundamental graphs of planar graphs.
We remark that our proof only uses the fact that each bipartite circle graph is the  fundamental graph of a graph $G$. Perhaps exploiting the planarity of $G$ could lead to an improved bound for the average degree.

Finally, let us remark that we have encountered a natural notion for degree-boundedness for matroids.
We say that a class of matroids is \emph{degree-bounded} if their $K_{s,s}$-free fundamental graphs have bounded average degree.
With this notion, observe that \cref{lem:fun} states that graphic (and cographic) matroids are degree-bounded.
Binary matroids are not degree-bounded since each bipartite graph is the fundamental graph of some binary matroid.
However, since fundamental graphs of a proper minor-closed class of binary matroids are a proper pivot-minor-closed class of bipartite graphs, we obtain the following from \cref{main} (indeed we implicitly proved this in order to prove \cref{main}).

\begin{theorem}
    Every proper minor-closed class of binary matroids is degree-bounded.
\end{theorem}

It could be interesting to further investigate degree-boundedness for matroids.

\begin{remark}
    \cref{odd} has been proven very recently by Zach Hunter \cite{hunter2025odd}.
    Thus, \cref{main} holds independently of \cref{thm:mat_mc_str}.
\end{remark}

\section*{Acknowledgements}

This work was done at the 2023 Vertex-Minor Workshop hosted by the Institute for Basic Science in South Korea.
We are grateful to the organizers and participants for providing an excellent research environment.
This work was supported by the Institute for Basic Science (IBS-R029-C1).
We thank the referees for helpful comments and in particular for pointing out an error with \cref{lem:rankconnect1} in a previous version of this paper.

\bibliographystyle{amsplain}
\bibliography{pivot}	

@article{du2024survey,
  title={A survey of degree-boundedness},
  author={Du, Xiying and McCarty, Rose},
  journal={European Journal of Combinatorics},
  pages={104092},
  year={2024},
  publisher={Elsevier}
}

@article{bourneuf2023polynomial,
  title={On polynomial degree-boundedness},
  author={Bourneuf, Romain and Buci{\'c}, Matija and Cook, Linda and Davies, James},
  journal={Advances in Combinatorics},
	volume={2024},
	pages={5, 16 pp.},
	year={2024},
	doi={10.19086/aic.2024.5},
}

@article{hunter2024k,
  title={K{\H{o}}v{\'a}ri-{S}{\'o}s-{T}ur{\'a}n theorem for hereditary families},
  author={Hunter, Zach and Milojevi{\'c}, Aleksa and Sudakov, Benny and Tomon, Istv{\'a}n},
  journal={Journal of Combinatorial Theory, Series B},
  volume={172},
  pages={168--197},
  year={2025},
  publisher={Elsevier}
}

@article{kuhn2004induced,
  title={Induced subdivisions in ${K_{s,s}}$-free graphs of large average degree},
  author={Kuhn, Daniela and Osthus, Deryk},
  journal={Combinatorica},
  volume={24},
  number={2},
  pages={287--304},
  year={2004},
  publisher={Springer-Verlag GmbH}
}

@article{girao2023induced,
  title={Induced subdivisions in ${K_{s,s}}$-free graphs with polynomial average degree},
  author={Gir{\~a}o, Ant{\'o}nio and Hunter, Zach},
  journal={International Mathematics Research Notices},
  volume={2025},
  number={4},
  pages={rnaf025},
  year={2025},
  publisher={Oxford University Press}
}

@article{du2023induced,
  title={Induced {$C_4$}-free subgraphs with large average degree},
  author={Du, Xiying and Gir{\~a}o, Ant{\'o}nio and Hunter, Zach and McCarty, Rose and Scott, Alex},
  journal={Journal of Combinatorial Theory, Series B},
  volume={173},
  pages={305--328},
  year={2025},
  publisher={Elsevier}
}

@article{davies2023pivot,
  title={Pivot-minors and the {E}rd{\H{o}}s-{H}ajnal conjecture},
  author={Davies, James},
  journal={Journal of Combinatorial Theory, Series B},
  volume={173},
  pages={257--278},
  year={2025},
  publisher={Elsevier}
}

@article{de1981local,
  title={Local complementation and interlacement graphs},
  author={De Fraysseix, Hubert},
  journal={Discrete Mathematics},
  volume={33},
  number={1},
  pages={29--35},
  year={1981},
  publisher={Elsevier}
}

@article{geelen2009circle,
  title={Circle graph obstructions under pivoting},
  author={Geelen, Jim and Oum, Sang-il},
  journal={Journal of Graph Theory},
  volume={61},
  number={1},
  pages={1--11},
  year={2009},
  publisher={Wiley Online Library}
}

@article{geelen2015highly,
  title={The highly connected matroids in minor-closed classes},
  author={Geelen, Jim and Gerards, Bert and Whittle, Geoff},
  journal={Annals of Combinatorics},
  volume={19},
  number={1},
  pages={107--123},
  year={2015},
  publisher={Springer}
}

@article{kim2024vertex,
  title={Vertex-minors of graphs: A survey},
  author={Kim, Donggyu and Oum, Sang-il},
  journal={Discrete Applied Mathematics},
  volume={351},
  pages={54--73},
  year={2024},
  publisher={Elsevier}
}

@inproceedings{mader1972existenz,
  title={Existenz n-fach zusammenh{\"a}ngender Teilgraphen in Graphen gen{\"u}gend grosser Kantendichte},
  author={Mader, Wolfgang},
  booktitle={Abhandlungen aus dem mathematischen Seminar der Universit{\"a}t Hamburg},
  volume={37},
  pages={86--97},
  year={1972},
  organization={Springer}
}

@article{mccarty2021dense,
  title={Dense induced subgraphs of dense bipartite graphs},
  author={McCarty, Rose},
  journal={SIAM Journal on Discrete Mathematics},
  volume={35},
  number={2},
  pages={661--667},
  year={2021},
  publisher={SIAM}
}

@article{gyarfas1980induced,
  title={Induced subtrees in graphs of large chromatic number},
  author={Gy{\'a}rf{\'a}s, Andr{\'a}s and Szemeredi, Endre and Tuza, Zs},
  journal={Discrete Mathematics},
  volume={30},
  number={3},
  pages={235--244},
  year={1980},
  publisher={Elsevier}
}

@article{kierstead1996applications,
  title={Applications of hypergraph coloring to coloring graphs not inducing certain trees},
  author={Kierstead, Hal A and Rodl, V},
  journal={Discrete Mathematics},
  volume={150},
  number={1-3},
  pages={187--193},
  year={1996},
  publisher={Elsevier}
}

@article{oum2009excluding,
  title={Excluding a bipartite circle graph from line graphs},
  author={Oum, Sang-il},
  journal={Journal of Graph Theory},
  volume={60},
  number={3},
  pages={183--203},
  year={2009},
  publisher={Wiley Online Library}
}

@article{WIMatroid,
    author={Oxley, James},
    title ={What is a Matroid?},
    journal={https://www.math.lsu.edu/\string~oxley/survey4.pdf}
}

@book{Oxley,
    AUTHOR = {Oxley, James},
     TITLE = {Matroid theory},
    SERIES = {Oxford Graduate Texts in Mathematics},
    VOLUME = {21},
   EDITION = {Second},
 PUBLISHER = {Oxford University Press, Oxford},
      YEAR = {2011},
     PAGES = {xiv+684},
      ISBN = {978-0-19-960339-8},
   MRCLASS = {05-01 (05B35 90C27)},
MRREVIEWER = {Maruti\ M.\ Shikare},
       DOI = {10.1093/acprof:oso/9780198566946.001.0001},
       URL = {https://doi.org/10.1093/acprof:oso/9780198566946.001.0001},
}

@article{scott2023polynomial,
  title={Polynomial bounds for chromatic number. {I}. {E}xcluding a biclique and an induced tree},
  author={Scott, Alex and Seymour, Paul and Spirkl, Sophie},
  journal={Journal of Graph Theory},
  volume={102},
  number={3},
  pages={458--471},
  year={2023},
  publisher={Wiley Online Library}
}

@article{fox2010separator,
  title={A separator theorem for string graphs and its applications},
  author={Fox, Jacob and Pach, J{\'a}nos},
  journal={Combinatorics, Probability and Computing},
  volume={19},
  number={3},
  pages={371--390},
  year={2010},
  publisher={Cambridge University Press}
}

@article{scott2020survey,
  title={A survey of $\chi$-boundedness},
  author={Scott, Alex and Seymour, Paul},
  journal={Journal of Graph Theory},
  volume={95},
  number={3},
  pages={473--504},
  year={2020},
  publisher={Wiley Online Library}
}

@article{hunter2025odd,
  title={Induced subdivisions of certain length},
  author={Hunter, Zach},
  journal={(In Preparation)}
}

@article{char2025edge,
  title={Edge-colouring and orientations: applications to degree-boundedness and $\chi$-boundedness},
  author={Char, Arnab and Kawarabayashi, Ken-ichi and Picasarri-Arrieta, Lucas},
  journal={arXiv preprint arXiv:2506.23054},
  year={2025}
}

@article{hunter2025c_4,
  title={${C}_4$-free subgraphs of high degree with geometric applications},
  author={Hunter, Zach and Milojevi{\'c}, Aleksa and Sudakov, Benny and Tomon, Istvan},
  journal={arXiv preprint arXiv:2506.23942},
  year={2025}
}

@article{tutte1959matroids,
  title={Matroids and graphs},
  author={Tutte, William Thomas},
  journal={Transactions of the American Mathematical Society},
  volume={90},
  number={3},
  pages={527--552},
  year={1959},
  publisher={JSTOR}
}

@article{grace2018perturbations,
  title={On perturbations of highly connected dyadic matroids},
  author={Grace, Kevin and Van Zwam, Stefan HM},
  journal={Annals of Combinatorics},
  volume={22},
  number={3},
  pages={513--542},
  year={2018},
  publisher={Springer}
}

@article{oum2005rank,
  title={Rank-width and vertex-minors},
  author={Oum, Sang-il},
  journal={Journal of Combinatorial Theory, Series B},
  volume={95},
  number={1},
  pages={79--100},
  year={2005},
  publisher={Elsevier}
}

@inproceedings{bouchet1989connectivity,
  title={Connectivity of isotropic systems},
  author={Bouchet, Andr{\'e}},
  booktitle={Proceedings of the third international conference on Combinatorial mathematics},
  pages={81--93},
  year={1989}
}

@article{bouchet1988graphic,
  title={Graphic presentations of isotropic systems},
  author={Bouchet, Andr{\'e}},
  journal={Journal of Combinatorial Theory, Series B},
  volume={45},
  number={1},
  pages={58--76},
  year={1988},
  publisher={Elsevier}
}

\end{document}